\RequirePackage{fix-cm}
\documentclass[smallextended]{svjour3}       
\smartqed  
\usepackage{float}
\usepackage{cite}%
\usepackage{color}
\usepackage{graphicx}
\usepackage{amscd, amsmath, amssymb}
\usepackage{pgf,tikz}
\usetikzlibrary{arrows}
\usepackage{pgfplots}
\usepackage{enumerate}
\usepackage{caption}
\usepackage{subcaption}
\usepackage{tikz}
\usepackage{pgfplots}
\usepackage{adjustbox}
\usepackage{hyperref}
\usepackage{todonotes}

\usepackage{tikz}
\usetikzlibrary{shapes}
\usetikzlibrary{plotmarks}
\usetikzlibrary{arrows,automata}
\usetikzlibrary{arrows.meta}
\usepackage{amsmath}
\usepackage{amssymb}

\usepackage{amsthm}
\usepackage{amsfonts}
\usepackage{enumerate}
\usepackage{mathtools} 
\usepackage{xcolor}
\usepackage{rotating,booktabs,multirow}
\usepackage{varwidth}
\usepackage{algorithm,algpseudocode}
\algnewcommand{\Inputs}[1]{%
  \State \textbf{Inputs:}
  \Statex \hspace*{\algorithmicindent}\parbox[t]{.8\linewidth}{\raggedright #1}
}
\algnewcommand{\Initialize}[1]{%
  \State \textbf{Initialization:}
  \Statex \hspace*{\algorithmicindent}\parbox[t]{.8\linewidth}{\raggedright #1}
}

\DeclareMathOperator{\st}{s.t.}
\DeclareMathOperator{\tr}{tr}
\usepackage{algpseudocode}

\algnewcommand{\algorithmicand}{\textbf{ and }}
\algnewcommand{\algorithmicor}{\textbf{ or }}
\algnewcommand{\OR}{\algorithmicor}
\algnewcommand{\AND}{\algorithmicand}
\algnewcommand{\var}{\texttt}

\begin{document}

\title{Conditions for linear convergence of the gradient method for non-convex optimization\thanks{This work was supported by the
 Dutch Scientific Council (NWO)  grant OCENW.GROOT.2019.015, \emph{Optimization for and with Machine Learning (OPTIMAL)}.}
}

\titlerunning{Convergence rate of gradient method}        

\author{ Hadi Abbaszadehpeivasti
  \and
            Etienne de Klerk
            \and  Moslem Zamani}


\institute{
H. Abbaszadehpeivasti\at Tilburg University, Department of Econometrics and Operations Research, Tilburg, The Netherlands\\
\email{h.abbaszadehpeivasti@tilburguniversity.edu}
\and
E. De Klerk\at Tilburg University, Department of Econometrics and Operations Research, Tilburg, The Netherlands\\
\email{e.deklerk@tilburguniversity.edu}
\and
M. Zamani\at Tilburg University, Department of Econometrics and Operations Research, Tilburg, The Netherlands\\
\email{m.zamani\_1@tilburguniversity.edu}
}

\date{Received: date / Accepted: date}

\maketitle

\begin{abstract}
In this paper, we derive a new linear convergence rate for the gradient method  with fixed step lengths for non-convex smooth optimization problems satisfying the Polyak-\L ojasiewicz (P\L) inequality.  We establish that the P\L\ inequality is a necessary and sufficient condition for linear convergence to the optimal value for this class of problems. We list some related classes of functions for which the gradient method may enjoy linear convergence rate. Moreover, we investigate their relationship with the P\L\ inequality.
\keywords{Weakly convex optimization \and Gradient method  \and Performance estimation problem \and Polyak-\L ojasiewicz inequality \and Semidefinite programming}
\end{abstract}

\section{Introduction}
We consider the gradient method for the unconstrained optimization problem
\begin{align}\label{P}
f^\star:=\inf_{x\in \mathbb{R}^n} f(x),
\end{align}
where $f: \mathbb{R}^n\to \mathbb{R}$ is differentiable, and $f^\star$ is finite. The gradient method with fixed step lengths may be described as follows.
\begin{algorithm}
\caption{Gradient method  with fixed step lengths}
\begin{algorithmic}
\State Set $N$ and $\{t_k\}_{k=1}^N$ (step lengths) and pick $x^1\in\mathbb{R}^n$.
\State For $k=1, 2, \ldots, N$ perform the following step:\\
\begin{enumerate}
\item
$x^{k+1}=x^k-t_k\nabla f(x^k)$
\end{enumerate}
\end{algorithmic}
\label{Alg1}
\end{algorithm}

 In addition, we assume that $f$ has a maximum curvature $L\in (0, \infty)$ and  a minimum curvature $\mu\in (-\infty, L)$. Recall that $f$ has a maximum curvature $L$ if $\tfrac{L}{2}\| .\|^2-f$ is convex. Similarly, $f$ has a  minimum curvature $\mu$ if $f-\tfrac{\mu}{2}\| .\|^2$ is convex. We denote smooth functions with curvature
belonging to the interval $[\mu, L]$ by $\mathcal{F}_{\mu, L}(\mathbb{R}^n)$. The class $\mathcal{F}_{\mu, L}(\mathbb{R}^n)$ includes all smooth functions with Lipschitz gradient (note that $\mu \ge 0$ corresponds to convexity). Indeed, $f$ is $L$-smooth on $\mathbb{R}^n$ if and only if $f$ has a maximum and minimum curvature $\bar L>0$ and $\bar \mu$, respectively, with $\max(\bar L, |\bar\mu|)\leq L$. This class of functions is broad and appears naturally in many models in machine learning, see \cite{Davis} and the references therein.

For $f\in\mathcal{F}_{\mu, L}(\mathbb{R}^n)$, we have the following inequalities for $x, y\in\mathbb{R}^n$
\begin{align}
 f(y)\leq f(x)+\langle \nabla f(x), y-x\rangle+\tfrac{L}{2}\|y-x\|^2,\label{L_E1}\\
 f(y)\geq f(x)+\langle \nabla f(x), y-x\rangle+\tfrac{\mu}{2}\|y-x\|^2;\label{L_E2}
\end{align}
see Lemma 2.5 in \cite{TeodorHypo}.

It is known that the lower bound of first order methods for obtaining an  $\epsilon$-stationary point is of the order  $\Omega\left(\epsilon^{-2}\right)$ for $L$-smooth functions \cite{carmon2019lower}. Hence, it is of interest to investigate the classes of functions for which the gradient method enjoys  linear convergence rate. This subject has been investigated by some scholars and some classes of functions have been introduced where linear convergence is possible; see \cite{hinder2020near, Karimi, hu2021analysis, danilova2020recent} and the references therein. This includes the class of functions satisfying the Polyak-\L ojasiewicz (P\L) inequality \cite{Karimi, polyak1963gradient}.

 \begin{definition}\label{Def_PL}
A function $f$ is said to satisfy the P\L\ inequality on $X\subseteq \mathbb{R}^n$ if there exists $\mu_p>0$ such that
\begin{align}\label{PL}
f(x)-f^\star \leq \tfrac{1}{2\mu_p} \| \nabla f(x)\|^2, \ \ \ \forall x\in X.
\end{align}
\end{definition}
Note that the P\L\ inequality is also known as \textit{ gradient dominated}; see \cite[Definition 4.1.3]{Nesterov}. The following classical theorem provides a linear convergence rate for Algorithm \ref{Alg1} under the P\L\ inequality.

\begin{theorem}{\cite[Theorem 4]{polyak1963gradient}}\label{T.PL.O}
Let $f$ be L-smooth and let $f$ satisfy P\L\ inequality on $X=\{x: f(x)\leq f(x^1)\}$. If  $t_1\in (0, \tfrac{2}{L})$ and $x^2$ is generated by Algorithm \ref{Alg1}, then
\begin{align}
f(x^2)-f^\star\leq \left(1-t_1L(2-t_1L_1)\tfrac{\mu_p}{L}\right)\left(f(x^1)-f^\star  \right).
\end{align}
 In particular, if $t_1=\tfrac{1}{L}$, we have
\begin{align}
f(x^2)-f^\star\leq (1-\tfrac{\mu_p}{L})\left(f(x^1)-f^\star\right).
\end{align}
\end{theorem}

In this paper we will sharpen this bound; see Theorem \ref{Th.PL}. Under the assumptions of Theorem \ref{T.PL.O},  Karimi et al. \cite{Karimi} established linear convergence rates for some other methods including the randomized coordinate descent. We refer the interested reader to the recent survey \cite{danilova2020recent} for more details on the convergence of non-convex algorithms under the P\L\ inequality.

In this paper, we analyse the gradient method from black-box perspective, which means that we have access to the gradient and the function value at the given point. Furthermore, we study the convergence rate of Algorithm \ref{Alg1} by using performance estimation.

In recent years performance estimation has been used to find worst-case convergence rates of first order methods \cite{abbaszadehpeivasti2021rate,abbaszadehpeivasti2021exact,gupta2022branch,taylor2017smooth,de2017worst,Taylor}, to name but a few. This strong tool first has been introduced by Drori and Teboulle in their seminal paper \cite{drori2014performance}. The idea of  performance estimation is that the infinite dimensional optimization problem concerning the computation of convergence rate may be formulated as a finite dimensional optimization problem (often semidefinite programs) by using interpolation theorems.

The rest of the paper is organized as follows. In Section \ref{PL_sec},  we consider problem \eqref{P} when $f$ satisfies the P\L\ inequality. We derive a new linear convergence rate for Algorithm \ref{Alg1} by using performance estimation.  Furthermore, we provide an optimal step length with respect to the given bound. We also show that the P\L \ inequality is necessary and sufficient for linear convergence, in a well-defined sense. Section \ref{quasar_sec} lists some other situations where Algorithm \ref{Alg1} is linearly convergent. Moreover, we study the relationships between these situations. Finally, we conclude the paper with some remarks and questions for future research.

\subsubsection*{Notation}
The $n$-dimensional Euclidean space is denoted by $\mathbb{R}^n$. Vectors are considered to be column vectors and the superscript $T$ denotes the transpose operation.
 We use $\langle \cdot, \cdot\rangle$ and $\| \cdot\|$ to denote the Euclidean inner product and norm, respectively.
  For a matrix $A$, $A_{ij}$ denotes its $(i, j)$-th entry. The  notation $A\succeq 0$ means the matrix $A$ is symmetric positive semi-definite, and $\tr(A)$ stands for the trace of $A$.

\section{ Linear convergence under the P\L\ inequality}\label{PL_sec}
 This section studies linear convergence of the gradient descent under the P\L\ inequality.
It is readily seen that the P\L\ inequality implies that every stationary point is a global minimum on $X$. By virtue of the descent lemma \cite[Page 29]{Nesterov}, we have
$$
f(x)-f^\star\geq \tfrac{1}{2L}\|\nabla f(x)\|^2, \ \forall x\in \mathbb{R}^n.
$$
Hence, $\mu_p$ can take value in $(0, L]$. On the other hand, we may assume without loss of generality  $\mu\leq \mu_p$.  The inequality is trivial if $\mu\leq 0$, and we therefore assume that  $\mu>0$. By taking the minimum with respect to $y$ from both side of inequality \eqref{L_E2}, we get
\[
f(x)-f^\star\leq \tfrac{1}{2\mu} \| \nabla f(x)\|^2.
\]
Hence, one may assume without loss of generality $\mu_p=\max\{\mu, \mu_p\}$ in inequality \eqref{PL}.

In what follows, we employ performance estimation to get a new bound under the assumptions of Theorem \ref{T.PL.O}.  In this setting, the worst-case convergence rate of Algorithm \ref{Alg1} may be cast as the following optimization problem,
\begin{align}\label{P1}
\nonumber   \max & \ \frac{f(x^2)-f^\star}{f(x^1)-f^\star}\\
&  \  x^2 \ \textrm{is generated  by Algorithm \ref{Alg1} w.r.t.}\ f, x^1  \\
\nonumber  & \ f(x)\geq f^\star \ \forall x\in\mathbb{R}^n\\
\nonumber  & \ f(x)-f^\star \leq \tfrac{1}{2\mu_p} \| \nabla f(x)\|^2, \ \ \ \forall x\in X\\
\nonumber  & \ f\in \mathcal{F}_{\mu, L}(\mathbb{R}^n)\\
\nonumber  & \ x^1\in\mathbb{R}^n.
\end{align}
 In problem \eqref{P1}, $f$ and $x^1$ are decision variables and $X=\{x: f(x)\leq f(x^1)\}$. We may replace the infinite dimensional condition $f \in \mathcal{F}_{\mu,L}(\mathbb{R}^n)$ by a finite set of constraints, by using interpolation. Theorem \ref{T1} gives some necessary and  sufficient conditions for the interpolation of given data by some  $f\in\mathcal{F}_{\mu, L}(\mathbb{R}^n)$.
\begin{theorem}{\cite[Theorem 3.1]{TeodorHypo}}\label{T1}
Let  $\{(x^i; g^i; f^i)\}_{i\in I}\subseteq \mathbb{R}^n\times \mathbb{R}^n \times \mathbb{R}$ with a given index set $I$ and let $L\in (0, \infty]$ and $\mu\in (-\infty, L)$. There
 exists a function  $f\in \mathcal{F}_{\mu, L}(\mathbb{R}^n)$ with
\begin{align}\label{int_fg}
f(x^i) = f^i, \nabla f(x^i) = g^i \ \ i\in I,
\end{align}
if and only if for every $i, j\in I$
{\small{
\begin{align}\label{Int-c}
\tfrac{1}{2(1-\tfrac{\mu}{L})}\left(\tfrac{1}{L}\left\|g^i-g^j\right\|^2+\mu\left\|x^i-x^j\right\|^2-\tfrac{2\mu}{L}\left\langle g^j-g^i,x^j-x^i\right\rangle\right)\leq f^i-f^j-\left\langle g^j, x^i-x^j\right\rangle.
\end{align}
}}
\end{theorem}
It is worth noting that Theorem \ref{T1} addresses non-smooth functions as well. In fact, $L=\infty$ covers non-smooth functions. Note that we only investigate the smooth case in this paper, that is $L\in (0, \infty)$ and $\mu\in (-\infty, 0]$.

By Theorem \ref{T1}, problem \eqref{P1} may be relaxed as follows,
\begin{align}\label{P2}
\nonumber   \max & \ \frac{f^2-f^\star}{f^1-f^\star}\\
\nonumber \st \ & \tfrac{1}{2(1-\tfrac{\mu}{L})}\left(\tfrac{1}{L}\left\|g^i-g^j\right\|^2+\mu\left\|x^i-x^j\right\|^2-\tfrac{2\mu}{L}\left\langle g^j-g^i,x^j-x^i\right\rangle\right)\leq\\
\nonumber& f^i-f^j-\left\langle g^j, x^i-x^j\right\rangle \ \ i, j\in\left\{1, 2\right\}  \\
&  \  x^{2}=x^1-t_1 g^1   \\
\nonumber& \ f^k\geq f^\star  \ \ k\in\{1, 2\}\\
\nonumber  & \ f^k-f^\star \leq \tfrac{1}{2\mu_p}\|g^k\|^2,  \ \ k\in\{1, 2\}.
\end{align}
As we replace the constraint  $ f(x)-f^\star \leq \tfrac{1}{2\mu_p} \| \nabla f(x)\|^2$ for each $x\in X$ by $ f^1-f^\star \leq \tfrac{1}{2\mu_p}\|g^1\|^2$ and $ f^2-f^\star \leq \tfrac{1}{2\mu_p}\|g^2\|^2$, problem \eqref{P2} is a relaxation of problem \eqref{P1}. By using the constraint $x^2=x^1-t_1g^1$, problem \eqref{P2} may be reformulated as,
\begin{align}\label{P3}
\nonumber   \max & \ \frac{f^2-f^\star}{f^1-f^\star}\\
\nonumber \st \ & \frac{1}{2(L-\mu)}\left(\|g^2\|^2+ (1+\mu Lt_1^2-2\mu t_1)\|g^1\|^2+2(\mu t_1-1)\langle g^1,g^2\rangle\right)- \\
\nonumber&f^2 + f^1 - \left\langle g^1,t_1g^1\right\rangle \leq 0\\
\nonumber & \frac{1}{2(L-\mu)}\left(\|g^2\|^2+ (1+\mu Lt_1^2-2\mu t_1)\|g^1\|^2+2(\mu t_1-1)\langle g^1,g^2\rangle\right)- \\
&f^1 + f^2 + \left\langle g^2,t_1g^1\right\rangle\leq 0  \\
\nonumber & \ f^\star-f^k\leq 0  \ \ k\in\left\{1, 2\right\}\\
\nonumber  & \ f^k-f^\star-\tfrac{1}{2\mu_p}\|g^k\|^2\leq 0,  \ \ k\in\left\{1, 2\right\}.
\end{align}
By using the Gram matrix,
\begin{align*}
  X=\begin{pmatrix}
      (g^1)^T \\
      (g^2)^T
    \end{pmatrix}\begin{pmatrix}
                   g^1 & g^2
                 \end{pmatrix}
  =\begin{pmatrix}
      \|g^1\|^1 & \langle g^1,g^2\rangle \\
      \langle g^1,g^2\rangle & \|g^2\|^2
    \end{pmatrix},
\end{align*}
problem \eqref{P3} can be relaxed as follows,
\begin{align}\label{P_gram}
\nonumber   \max & \ \frac{f^2-f^\star}{f^1-f^\star}\\
\nonumber \st \ & \tr(A_1X)-f^2 + f^1 \leq 0\\
& \tr(A_2X)-f^1 + f^2\leq 0  \\
\nonumber  & \ f^1-f^\star+\tr(A_3X)\leq 0 \\
\nonumber  & \ f^2-f^\star+\tr(A_4X)\leq 0\\
\nonumber& \ f^1,f^2\geq f^\star, X\succeq 0,
\end{align}
where
\begin{align*}
  &A_1=\begin{pmatrix}
        \frac{1+\mu Lt_1^2-2\mu t_1}{2(L-\mu)}-t_1 & \frac{\mu t_1-1}{2(L-\mu)} \\
        \frac{\mu t_1-1}{2(L-\mu)} & \frac{1}{2(L-\mu)}
      \end{pmatrix} \ \ \ \ \
  &&A_2=\begin{pmatrix}
        \frac{1+\mu Lt_1^2-2\mu t_1}{2(L-\mu)} & \frac{\mu t_1-1}{2(L-\mu)}+\frac{t_1}{2} \\
        \frac{\mu t_1-1}{2(L-\mu)}+\frac{t_1}{2} & \frac{1}{2(L-\mu)}
      \end{pmatrix} \\
  &A_3=\begin{pmatrix}
        \frac{-1}{\mu^2_p} & 0 \\
        0 & 0
        \end{pmatrix}\ \ \ \
  &&A_4=\begin{pmatrix}
        0 & 0 \\
        0 & \frac{-1}{\mu^2_p}
        \end{pmatrix}.
\end{align*}
In addition, $X, f^1, f^2$ are decision variables in this formulation.
In the next theorem, we obtain an upper bound for problem \eqref{P3} by using  weak duality. This bound gives a new convergence rate for Algorithm \ref{Alg1} for a wide variety of functions.

\begin{theorem}\label{Th.PL}
Let $f\in \mathcal{F}_{\mu, L}(\mathbb{R}^n)$ with $L\in(0, \infty), \mu\in (-\infty, 0]$ and let $f$ satisfy the P\L\ inequality on $X=\{x: f(x)\leq f(x^1)\}$. Suppose that $x^2$ is generated by Algorithm \ref{Alg1}.
\begin{enumerate}[i)]
\item
If $t_1\in\left(0,\tfrac{1}{L}\right)$, then
\begin{align*}
&\frac{f(x^2)-f^\star}{f(x^1)-f^\star}\leq\\
& \left(\frac{\mu_p\left(1-Lt_1\right)+\sqrt{\left(L-\mu\right) \left(\mu-\mu_p \right)\left(2-L t_1\right)\mu_pt_1+\left(L-\mu\right)^2}}{L-\mu+\mu_p}\right)^2.
\end{align*}
\item
If  $t_1\in \left[\tfrac{1}{L}, \tfrac{3}{\mu+L+\sqrt{\mu^2-L\mu+L^2}}\right]$, then
\begin{align*}
\frac{f(x^2)-f^\star}{f(x^1)-f^\star}\leq \left(\frac{(Lt_1-2)(\mu t_1-2)\mu_p t_1}{\left(L+\mu-\mu_p\right)t_1-2}+1\right).
\end{align*}
\item
If $t_1\in\left(\tfrac{3}{\mu+L+\sqrt{\mu^2-L\mu+L^2}}, \tfrac{2}{L}\right)$, then
\begin{align*}
\frac{f(x^2)-f^\star}{f(x^1)-f^\star}\leq\frac{(L t_1-1)^2}{(Lt_1-1)^2+\mu_p t_1(2-Lt_1)}.
\end{align*}
\end{enumerate}
In particular, if $t_1=\tfrac{1}{L}$ and $\mu=-L$, we have
\begin{align}\label{bound.1}
f(x^2)-f^\star\leq \left(\frac{2L-2\mu_p}{2L+\mu_p}\right)\left(f(x^1)-f^\star\right).
\end{align}
\end{theorem}
\begin{proof}
First we consider $t_1\in\left(0,\tfrac{1}{L}\right)$. Let
\begin{align*}
&b_1=\frac{\left(L-\mu\right)\left(\alpha+\mu_p\left(1-L t_1\right)\right)}{\alpha\left(L-\mu+\mu_p\right)}\\
&b_2=b_1-\left(\frac{\alpha}{L-\mu}b_1\right)^2,
\end{align*}
where
\[
\alpha=\sqrt{\left(L-\mu\right)\left(\mu_pt_1\left(\mu_p-\mu\right)\left(Lt_1-2\right)+\left(L-\mu\right)\right)}.
\]
It is readily seen that $b_1,b_2\geq 0$. Furthermore,
\begin{align*}
    & {f^2-f^\star}-(b_1-b_2)\left( {f^1-f^\star} \right)-b_2\left(-\frac{1}{2\mu_p}\left\|g^1\right\|^2+f^1-f^\star\right)\\
    &-\left(1-b_1\right)\left(-\frac{1}{2\mu_p}\left\|g^2\right\|^2+f^2-f^\star\right)-b_1\bigg(\frac{1}{2(L-\mu)}\big(\|g^2\|^2+\\
    &(1+\mu Lt_1^2-2\mu t_1)\|g^1\|^2+2(\mu t_1-1)\langle g^1,g^2\rangle\big) -f^1 + f^2 + \left\langle g^2,t_1g^1\right\rangle\bigg)=\\
    &-\frac{1-Lt_1}{2\alpha}\left\|\frac{\alpha+\mu_p\left(1-L t_1\right)}{L-\mu+\mu_p}g^1-g^2\right\|^2\leq 0.
\end{align*}
Therefore, for any feasible solution of problem \eqref{P3}, we have
\begin{align*}
&\frac{f(x^2)-f^\star}{f(x^1)-f^\star}\leq\\
& \left(\frac{\mu_p\left(1-Lt_1\right)+\sqrt{\left(L-\mu\right) \left(\mu-\mu_p \right)\left(2-L t_1\right)\mu_pt_1+\left(L-\mu\right)^2}}{L-\mu+\mu_p}\right)^2,
\end{align*}
and the proof of this part is complete. Now, we consider the case that \linebreak
$t_1\in \left[\tfrac{1}{L}, \tfrac{3}{\mu+L+\sqrt{\mu^2-L\mu+L^2}}\right]$. Suppose that
\begin{align*}
   & a_1=\frac{\mu t_1-1}{\left(L+\mu-\mu_p\right)t_1-2}, \ \ \ \
   && a_2=\frac{1-L t_1}{\left(L+\mu-\mu_p\right)t_1-2} \\
   & a_3=-\frac{\left((Lt_1-2)(\mu t_1-2)-1\right)\mu_pt_1}{\left(L+\mu-\mu_p\right)t_1-2} \ \ \ \
   && a_4=-\frac{\mu_p t_1}{\left(L+\mu-\mu_p\right)t_1-2}.
\end{align*}
 It is readily seen that $a_1,a_2,a_3,a_4\geq 0$. Furthermore,
\begin{align*}
    & {f^2-f^\star}-\left(1-a_3-a_4\right)\left( {f^1-f^\star} \right)-a_3\left(-\frac{1}{2\mu_p}\left\|g^1\right\|^2+f^1-f^\star\right)-\\
    &a_4\left(-\frac{1}{2\mu_p}\left\|g^2\right\|^2+f^2-f^\star\right)-a_1\bigg(\frac{1}{2(L-\mu)}\big(\|g^2\|^2+(1+\mu Lt_1^2-2\mu t_1)\|g^1\|^2+\\
    &2(\mu t_1-1)\langle g^1,g^2\rangle\big) -f^1 + f^2 + \left\langle g^2,t_1g^1\right\rangle\bigg)-a_2\bigg(\frac{1}{2(L-\mu)}\big(\|g^2\|^2+ \\
    &(1+\mu Lt_1^2-2\mu t_1)\|g^1\|^2+2(\mu t_1-1)\langle g^1,g^2\rangle\big) -f^2 + f^1 - \left\langle g^1,t_1g^1\right\rangle \bigg)=0.
\end{align*}
Therefore, for any feasible solution of problem \eqref{P3}, we have
$$
f(x^2)-f^\star- \left(\frac{L\mu_p\mu t_1^3-2\mu_p\left(L+\mu\right)t_1^2+4\mu_p t_1}{\left(L+\mu-\mu_p\right)t_1-2}+1 \right)\left(f(x^1)-f^\star  \right)\leq 0.
$$
Now, we prove the last part. Assume that $t_1\in\left(\tfrac{3}{\mu+L+\sqrt{\mu^2-L\mu+L^2}},\tfrac{2}{L}\right)$. With some algebra, one can show
{\small{
\begin{align*}
    & {f^2-f^\star}-\left(\frac{(L t_1-1)^2}{\beta}\right)\left( {f^1-f^\star} \right)-
    \left(\frac{\mu_p t_1(2-Lt_1)}{\beta}\right)\left(-\frac{1}{2\mu_p}\left\|g^2\right\|^2+f^2-f^\star\right)-\\
    &\left(\frac{(L t_1-1)(2-Lt_1)}{\beta}\right)\bigg(\frac{1}{2(L-\mu)}\big(\|g^2\|^2+(1+\mu Lt_1^2-2\mu t_1)\|g^1\|^2+2(\mu t_1-1)\langle g^1,g^2\rangle\big) -\\
    &f^2 + f^1 - \left\langle g^1,t_1g^1\right\rangle \bigg)-\left( \frac{L t_1-1}{\beta} \right)\bigg(\frac{1}{2(L-\mu)}\big(\|g^2\|^2+ (1+\mu Lt_1^2-2\mu t_1)\|g^1\|^2+\\
&2(\mu t_1-1)\langle g^1,g^2\rangle\big) -f^1 + f^2 + \left\langle g^2,t_1g^1\right\rangle\bigg)=\\
&-\frac{(1-L t_1) \left(\mu t_1(Lt_1-2)+2(1-Lt_1)+1\right)}{2 \beta(L-\mu)}\left\|\sqrt{Lt_1-1}g^1+\tfrac{1}{\sqrt{Lt_1-1}} g^2\right\|^2\leq 0,
\end{align*}
}}
where,
\begin{align*}
  \beta=(Lt_1-1)^2+\mu_p t_1(2-Lt_1).
\end{align*}
The rest of the proof is similar to that of the former cases.
\end{proof}

One may wonder how we obtain Lagrange multipliers in Theorem \ref{Th.PL}. The multipliers are computed by solving the dual of problem \eqref{P_gram} by hand. Furthermore, one can verify that Theorem \ref{Th.PL} provides a tighter bound in comparison with the convergence rate given in Theorem \ref{T.PL.O} for $L$-smooth functions with $t_1\in (0, \tfrac{2}{L})$. The next proposition gives the optimal step length with respect to the worst-case convergence rate.

\begin{proposition}\label{Por_opt_s}
  Let $f\in \mathcal{F}_{\mu, L}(\mathbb{R}^n)$ with $L\in(0, \infty), \mu\in (-\infty, 0]$ and let $f$ satisfy the P\L\ inequality on $X=\{x: f(x)\leq f(x^1)\}$. Suppose that $r(t)=L\mu(L+\mu-\mu_p)t^3- \left(L^2-\mu_p (L+\mu)+5 L \mu+\mu^2\right)t^2+4(L+\mu)t-4$ and $\bar t$ is the unique root of $r$ in $\left[\tfrac{1}{L}, \tfrac{3}{\mu+L+\sqrt{\mu^2-L\mu+L^2}}\right]$ if it exists. Then $t^\star$ given by
 \begin{align*}
  t^\star=
  \begin{cases}
    \bar t & \mbox{\ if\ } \bar t \mbox{\ exists}\\
    \tfrac{3}{\mu+L+\sqrt{\mu^2-L\mu+L^2}} & \mbox{otherwise},
  \end{cases}
 \end{align*}
  is the optimal step length for Algorithm \ref{Alg1} with respect to the worst-case convergence rate.
\end{proposition}

\begin{proof}
  To obtain an optimal step length, we need to solve the optimization problem $$\min_{t\in(0, \tfrac{2}{L})} h(t),$$ where $h$ is given by
  {\small{
  \begin{align*}
  h(t)=
  \begin{cases}
    \left(\frac{\mu_p\left(1-Lt\right)+\sqrt{\left(L-\mu\right) \left(\mu-\mu_p \right)\left(2-L t\right)\mu_p t+\left(L-\mu\right)^2}}{L-\mu+\mu_p}\right)^2  &  t\in\left(0,\tfrac{1}{L}\right) \\
    \frac{(Lt-2)(\mu t-2)\mu_p t}{\left(L+\mu-\mu_p\right)t-2}+1 & t\in \left[\tfrac{1}{L}, \tfrac{3}{\mu+L+\sqrt{\mu^2-L\mu+L^2}}\right] \\
    \frac{(L t-1)^2}{(Lt-1)^2+(2-Lt)\mu_p t}  & t\in\left(\tfrac{3}{\mu+L+\sqrt{\mu^2-L\mu+L^2}}, \tfrac{2}{L}\right).
  \end{cases}
  \end{align*}
  }}
   It is easily seen that $h$ is decreasing on $\left(0,\tfrac{1}{L}\right)$ and is increasing on \linebreak
   $\left(\tfrac{3}{\mu+L+\sqrt{\mu^2-L\mu+L^2}}, \tfrac{2}{L}\right)$. Hence, we need investigate the closed interval \linebreak
   $\left[\tfrac{1}{L}, \tfrac{3}{\mu+L+\sqrt{\mu^2-L\mu+L^2}}\right]$. We will show that $h$ is convex on the interval in question. First, we consider the case $L+\mu-\mu_p\leq 0$. Let $p(t)=\frac{\mu t-2}{\left(L+\mu-\mu_p\right)t-2}$ and $q(t)=(Lt-2)\mu_p t$.
    By some algebra, one can show the following inequalities for $t\in\left[\tfrac{1}{L}, \tfrac{3}{\mu+L+\sqrt{\mu^2-L\mu+L^2}}\right]$:
   \begin{align*}
 & p(t)\geq 0 \ && q(t)\leq 0 \\
 & p^\prime(t)\geq 0 \ && q^\prime(t)\geq 0 \\
  & p^{\prime \prime}(t)\leq 0 \ && q^{\prime \prime}(t)\geq 0.
   \end{align*}
Hence, the convexity of $h$ follows from $h^{\prime \prime}=p^{\prime \prime}q+2p^{\prime}q^{\prime}+pq^{\prime \prime}$. Now, we investigate the case that $L+\mu-\mu_p>0$. Suppose that $p(t)=\frac{\mu_p t}{\left(L+\mu-\mu_p\right)t-2}$ and $q(t)=(Lt-2)(\mu t-2)$. For these functions, we have the following inequalities
  \begin{align*}
 & p(t)\leq 0 \ && q(t)\geq 0 \\
 & p^\prime(t)\leq 0 \ && q^\prime(t)\leq 0 \\
  & p^{\prime \prime}(t)\geq 0 \ && q^{\prime \prime}(t)\leq 0,
   \end{align*}
   which analogous to the former case one can infer the convexity of $h$ on the given interval.
Hence, if $h$ has a root in $\left[\tfrac{1}{L}, \tfrac{3}{\mu+L+\sqrt{\mu^2-L\mu+L^2}}\right]$, it will be the minimum. Otherwise, the point $t^\star=\tfrac{3}{\mu+L+\sqrt{\mu^2-L\mu+L^2}}$ will be the minimum. This follows from the point that $h^\prime(\tfrac{1}{L})=\tfrac{2L\mu_p(\mu_p-L)}{(L+\mu_p-\mu)^2}\leq 0$ and the convexity of $h$ on the interval in question.
\end{proof}

Thanks to Proposition \ref{Por_opt_s},  the following corollary gives the optimal step length  for $L$-smooth convex functions satisfying the P\L\ inequality.

\begin{corollary}
 If $f$ is an $L$-smooth convex function satisfying the P\L\ inequality, then the optimal step length with respect to the worst-case convergence rate is \linebreak
 $\min\left\{\tfrac{2}{L+\sqrt{L\mu_p}}, \tfrac{3}{2L}\right\}$.
\end{corollary}

The constant $\tfrac{2}{L+\sqrt{L\mu_p}}$ also appears in the  the  fast gradient algorithm introduced in \cite{necoara2019linear} for $L$-smooth convex functions which are $(1, \mu_s)$-quasar-convex, see Definition \ref{Def.Qs}. By Theorem \ref{rel2}, $(1, \mu_s)$-quasar-convexity implies the P\L\ inequality with the same constant.  Algorithm \ref{Alg2} describes the method in question.
\begin{algorithm}
\caption{Fast gradient method}
\begin{algorithmic}
\State Pick $x^1\in\mathbb{R}^n$, set $N$ and $y^1=x^1$.
\State For $k=1, 2, \ldots, N$ perform the following step:\\
\begin{enumerate}
\item
$y^{k+1}=x^k-\tfrac{1}{L}\nabla f(x^k)$
\item
$x^{k+1}=y^{k+1}+\tfrac{\sqrt{L}-\sqrt{\mu_p}}{\sqrt{L}+\sqrt{\mu_p}}\left(y^{k+1}-y^k \right)$
\end{enumerate}
\end{algorithmic}
\label{Alg2}
\end{algorithm}

One can verify that Algorithm \ref{Alg2}, at the first iteration, generates $x^2=x^1-\tfrac{2}{L+\sqrt{L\mu_p}}\nabla f(x^1)$.

A more general form of the P\L\ inequality, called the \L ojasiewicz inequality, may be written as
\begin{align}\label{L}
\left  (f(x)-f^\star \right)^{2\theta} \leq \tfrac{1}{2\mu_p} \| \nabla f(x)\|^2, \ \ \ \forall x\in X,
\end{align}
where $\theta\in (0, 1)$. It is known that when $\theta\in (0, \tfrac{1}{2}]$ some algorithms, including Algorithm \ref{Alg1}, is linearly convergent; see \cite{attouch2009convergence,attouch2010proximal}. In the next theorem, we show that  for functions with finite maximum and minimum curvature the \L ojasiewicz inequality cannot hold for $\theta\in (0, \tfrac{1}{2})$.

\begin{theorem}
Let $f\in \mathcal{F}_{\mu, L}(\mathbb{R}^n)$ be a non-constant function. If $f$ satisfies the \L ojasiewicz inequality on $X=\{x: f(x)\leq f(x^1)\}$, then  $\theta\geq \tfrac{1}{2}$.
\end{theorem}
\begin{proof}
To the contrary, assume that  $\theta\in(0, \tfrac{1}{2})$. Without loss of generality, we may assume that $\mu=-L$. It is known that Algorithm \ref{Alg1} generates a decreasing sequence $\{f(x^k)\}$ and it is convergent, that is $\|\nabla f(x^k)\|\to 0$; see \cite[page 28]{Nesterov}. Furthermore, \eqref{L} implies that $f(x^k)\to f^\star$. Without loss of generality, we may assume that $f^\star=0$. First, we investigate the case that $f(x^1)=1$. The semi-definite programming problem corresponding to performance estimation in this case may be formulated as follows,
\begin{align}\label{P_4}
\nonumber   \max & \ f^2\\
\nonumber \st \ & \tr(A_1X)-f^2 + 1 \leq 0\\
& \tr(A_2X)-1 + f^2\leq 0  \\
\nonumber  & \ 1+\tr(A_3X)\leq 0 \\
\nonumber  & \ (f^2)^{2\theta}+\tr(A_4X)\leq 0\\
\nonumber& \ f^2\geq 0, X\succeq 0.
\end{align}   Since  Algorithm \ref{Alg1} is a monotone method, $f^2$ can take value in $[0,1]$. In addition, we have $f^2\leq (f^2)^{2\theta}$ on this interval. Hence, by using Theorem \ref{Th.PL}, we get the following bound,
\[
f^2\leq \frac{2L-2\mu_p}{2L+\mu_p}.
\]
 Now, suppose that $f(x^1)=f^1>0$. Consider the function $h: \mathbb{R}^n\to\mathbb{R}$ given by $h(x)=\tfrac{f(x)}{f^1}$. It is seen that $h$ is $\tfrac{L}{f^1}$-smooth and
 \[
h(x)^{2\theta} \leq \tfrac{1}{2\mu_p(f^1)^{2\theta-2}} \| \nabla h(x)\|^2, \ \ \ \forall x\in X.
 \]
 As Algorithm \ref{Alg1} generates the same $x^2$ for both functions, by using the first part, we obtain
 \[
\frac{f(x^2)}{f(x^1)}\leq \frac{2L(f^1)^{-1}-2\mu_p(f^1)^{2\theta-2}}{2L(f^1)^{-1}+\mu_p(f^1)^{2\theta-2}}=\frac{2L-2\mu_p(f^1)^{2\theta-1}}{2L+\mu_p(f^1)^{2\theta-1}}.
\]
For $f^1$ sufficiently small, we have $\frac{2L-2\mu_p(f^1)^{2\theta-1}}{2L+\mu_p(f^1)^{2\theta-1}}< 0$, which contradicts $f^\star\geq 0$ and  the proof is complete.
\end{proof}

Necoara et al.  gave necessary and sufficient conditions for linear convergence of the gradient method with constant step lengths when $f$ is a smooth convex function; see \cite[Theorem 13]{necoara2019linear}. Indeed, the theorem says that Algorithm \ref{Alg1} is linearly convergent if and only if $f$ has a quadratic functional
growth on $\{x: f(x)\leq f(x^1)\}$; see Definition \ref{def3}. However, this theorem does not hold necessarily for non-convex functions. The next theorem provides necessary and sufficient conditions for  linear convergence of Algorithm \ref{Alg1}.

\begin{theorem}
Let $f\in \mathcal{F}_{\mu, L}(\mathbb{R}^n)$. Algorithm \ref{Alg1} is linearly convergent to the optimal value  if and only if $f$ satisfies P\L\ inequality on $\{x: f(x)\leq f(x^1)\}$.
\end{theorem}
\begin{proof}
Let $\bar x\in \{x: f(x)\leq f(x^1)\}$. Linear convergence implies the existence of $\gamma\in [0, 1)$ with
\begin{align}\label{c_c_c}
f(\hat x)-f^\star \leq \gamma\left( f(\bar x)-f^\star\right),
\end{align}
where $\hat x=\bar x-\tfrac{1}{L}\nabla f(\bar x)$. By \eqref{L_E2}, we have $f(\bar x)- f(\hat x)\leq \tfrac{2L-\mu}{2L^2}\|\nabla f(\bar x)\|^2$. By using this inequality with \eqref{c_c_c}, we get
\begin{align*}
f(\bar x)-f^\star \leq \tfrac{1}{1-\gamma} \left( f(\bar x)-f(\hat x)\right)\leq \tfrac{2L-\mu}{2L^2(1-\gamma)}\|\nabla f(\bar x)\|^2,
\end{align*}
which shows that $f$ satisfies P\L\ inequality on $\{x: f(x)\leq f(x^1)\}$. The other implication follows from Theorem \ref{Th.PL}.
\end{proof}

\section{The P\L\ inequality: relation to some classes of functions}\label{quasar_sec}

In this section, we study some classes of functions for which Algorithm \ref{Alg1} may be linearly convergent.  We establish that these classes of functions satisfy the P\L\ inequality under mild assumptions, and we infer the linear convergence by using Theorem \ref{Th.PL}. Moreover, one can get convergence rates by applying performance estimation.

Throughout the section, we denote the optimal solution set of problem \eqref{P} by $X^\star$ and we assume that $X^\star$ is non-empty. We denote the distance function to $X^\star$ by $d_{X^\star}(x):=\inf_{y\in X^\star} \|y-x\|$. The set-valued mapping $\Pi_{X^\star} (x)$ stands for the projection of $x$  on $X^\star$, that is,
$\Pi_{X^\star} (x)=\{y: \|y-x\|=d_{X^\star}(x)\}$.

\begin{definition}
Let $\mu_g> 0$. A function $f$ has a quadratic gradient growth on $X\subseteq \mathbb{R}^n$  if
\begin{align}\label{Qq_r}
\langle \nabla f(x), x-x^\star\rangle \geq \mu_g d_{X^\star}^2 (x), \ \ \ \forall x\in X,
\end{align}
for some $x^\star\in \Pi_{X^\star}(x)$.
\end{definition}
Note that inequality \eqref{L_E1} implies that $\mu_g\leq L$. Hu et al. \cite{hu2021analysis} investigated the convergence rate $\{x^k\}$ when $f$ satisfies \eqref{Qq_r} and $X^\star$ is singleton. To the best knowledge of the authors, there is no convergence rate result in terms of $\{f(x^k)\}$  for functions with a quadratic gradient growth. The next proposition states that quadratic gradient growth property implies the P\L\  inequality.

\begin{proposition}\label{Pr.1}
Let $f\in\mathcal{F}_{\mu, L}(\mathbb{R}^n)$. If $f$ has a quadratic gradient growth on $X\subseteq \mathbb{R}^n$ with $\mu_g> 0$, then $f$ satisfies the P\L\ inequality with $ \mu_p=\tfrac{\mu_g^2}{L}$.
\end{proposition}
\begin{proof}
Suppose that $x^\star\in \Pi_{X^\star}(x)$ satisfies \eqref{Qq_r}. By the Cauchy-Schwarz inequality, we have
\begin{align}\label{C-S}
\mu_g \|x-x^\star\|\leq \|\nabla f(x)\|.
\end{align}
On the other hand, \eqref{L_E1} implies that
\begin{align}\label{L-CC}
 f(x)\leq f(x^\star)+\tfrac{L}{2} \|x-x^\star\|^2.
\end{align}
The P\L\ inequality follows from \eqref{C-S} and \eqref{L-CC}.
\end{proof}
By Proposition \ref{Pr.1} and Theorem \ref{Th.PL}, one can infer the linear convergence of Algorithm \ref{Alg1} when $f$ has a quadratic gradient growth on $X=\{x: f(x)\leq f(x^1)\}$.
Indeed, one can derive the following bound if $t_1=\tfrac{1}{L}$ and $\mu=-L$,
\begin{align}\label{q_gr_bound1}
  f(x^2)-f^\star\leq\left(\frac{2L^2-2\mu_g^2}{2L^2+\mu_g^2}\right)\left(f(x^1)-f^\star\right).
\end{align}
Nevertheless, by using the performance estimation method, one can derive a better bound than the bound given by \eqref{q_gr_bound1}. The performance estimation problem for $t_1=\tfrac{1}{L}$ in this case may be formulated as
\begin{align}\label{QG_PEP}
\nonumber   \max & \ \frac{f^{2}-f^\star}{f^1-f^\star}\\
\nonumber \st \ & \{x^k, g^k, f^k\}\cup \{y^k, 0, f^\star\} \ \text{satisfy interpolation constraints \eqref{Int-c} for}\ k\in\{1, 2\}\\
&  \  x^{2}=x^1-\tfrac{1}{L} g^1  \\
\nonumber& \ f^k\geq f^\star  \ \  k\in\{1, 2\}\\
\nonumber  & \ \langle g^k, x^k-y^k\rangle \geq \mu_g \|y^k-x^k\|^2,  \ \  k\in\{1, 2\}\\
\nonumber  & \ \|x^1-y^1\|^2 \leq \|x^1-y^2\|^2\\
\nonumber  & \ \|x^2-y^2\|^2 \leq \|x^2-y^1\|^2.
\end{align}
Analogous to Section \ref{PL_sec}, one can obtain an upper bound for problem \eqref{QG_PEP} by solving a semidefinite program. Our numerical results show that the bounds generated by performance estimation is tighter than bound \eqref{q_gr_bound1}; see Figure \ref{fig1}. We do not have a closed-form bound on the optimal value of (21), though.
\begin{figure}
\begin{tikzpicture}
\begin{axis}[
    xmin = 0, xmax = 1,
    ymin = 0, ymax = 1,
    xtick distance = 0.5,
    ytick distance = 0.5,
    grid = both,
    minor tick num = 1,
    major grid style = {lightgray},
    minor grid style = {lightgray!50},
    width = \textwidth,
    height = 0.5\textwidth,
    xlabel = {$\frac{\mu_g}{L}$},
    ylabel = {Upper bound on function value},]

\addplot[
    domain = 0:1,
    samples = 1000,
    smooth,
    thick,
    blue,
] {(2-2*x*x)/(2+x*x)};

\addplot[
    smooth,
    thin,
    red,
    dashed
] table[col sep=comma]  {fig1.dat};
 \legend{
    Bound \eqref{q_gr_bound1},
    PEP bound \eqref{QG_PEP},
}

\end{axis}

\end{tikzpicture}
\caption{Convergence rate computed by performance estimation (red line)
and the bound given by \eqref{q_gr_bound1} (blue line)
for $\tfrac{\mu_g}{L}\in(0,1)$.}
\label{fig1}
\end{figure}
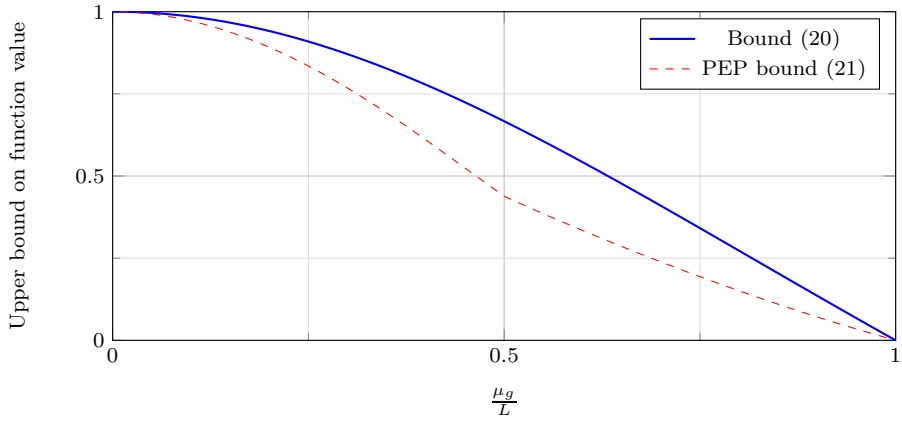

\begin{definition}{\cite[Definition 4]{necoara2019linear}, {\cite[Definition 4.1.2]{Nesterov}}}\label{def3}
Let $\mu_q> 0$. A function $f$ has a quadratic functional growth on $X\subseteq \mathbb{R}^n$ if
\begin{align}\label{QC_r}
\tfrac{\mu_q}{2}d^2_{X^\star}(x)  \leq f(x)-f^\star, \ \ \ \forall x\in X.
\end{align}
\end{definition}
It is readily seen that, contrary to the previous situations, the quadratic functional growth property does not necessarily imply that each stationary point is a global optimal solution. The next theorem investigates the relationship between quadratic functional growth property and other notions.

\begin{theorem}\label{Th.Q}
Let $f\in\mathcal{F}_{\mu, L}(\mathbb{R}^n)$ and let $X=\{x: f(x)\leq f(x^1)\}$. We have the following implications:
\begin{enumerate}[i)]
    \item
    \eqref{PL} $\Rightarrow$ \eqref{QC_r} with  $\mu_q=\mu_p$.
 \item
 If $\mu_q>\tfrac{-\mu L}{L-\mu}$, then  \eqref{QC_r} $\Rightarrow$ \eqref{Qq_r} with $\mu_g=\tfrac{\mu_q}{2}(1-\tfrac{\mu}{L})+\tfrac{\mu}{2}$.
  \item
 If
 $$
 f(x)-f(x^\star)\leq \langle \nabla f(x), x-x^\star\rangle, \ \forall x\in X,
 $$
for some $x^\star\in\Pi_{X^\star} (x)$ then  \eqref{QC_r} $\Rightarrow$ \eqref{Qq_r} with $\mu_g=\tfrac{\mu_q}{2}$.
\end{enumerate}
\end{theorem}
\begin{proof}
One can establish $i)$ similarly to the proof of  \cite[Theorem 2]{Karimi}. Consider part $ii)$. Let $x\in X$ and $x^\star\in\Pi_{X^\star} (x)$ with $d_{X^\star} (x)=\|x-x^\star\|$.  By \eqref{Int-c}, we have
 $$
 f(x)-f(x^\star)\leq \tfrac{-1}{2(L-\mu)}\| \nabla f(x)\|^2-\tfrac{\mu L}{2(L-\mu)}\|x-x^\star\|^2+\tfrac{L}{L-\mu}\langle \nabla f(x), x-x^\star\rangle.
 $$
 As $\tfrac{\mu_q}{2}\|x-x^\star\|^2 \leq  f(x)-f(x^\star)$, we get
 $$
\left(\tfrac{\mu_q}{2}(1-\tfrac{\mu}{L})+\tfrac{\mu}{2}\right)\left\|x-x^\star\right\|^2 \leq  \langle\nabla f(x), x-x^\star\rangle,
 $$
 which establishes the desired inequality. Part $iii)$ is proved similarly to the former case.
\end{proof}
By Theorem \ref{Th.PL}, it is clear that Algorithm \ref{Alg1} enjoys linear convergence rate if $f$ has a quadratic gradient growth on $X=\{x: f(x)\leq f(x^1)\}$ and if $f$ satisfies assumptions $ii)$ or $iii)$ in Theorem \ref{Th.Q}. For instance, if $\mu=-L$ and $\mu_q\in(\tfrac{L}{2}, L)$,  one can derive the following convergence rate for Algorithm \ref{Alg1} for fixed step length $t_k=\tfrac{1}{L}$, $k\in\{1, ..., N\}$,
\begin{align}\label{CR_Qg}
f(x^{N+1})-f(x^1) \leq\left(\frac{2L^2-2(\mu_q-\tfrac{L}{2})^2}{2L^2+(\mu_q-\tfrac{L}{2})^2}\right)^N \left(f(x^1)-f^\star  \right).
\end{align}
It is interesting to compare the convergence rate \eqref{CR_Qg} to the convergence rate obtained by using the performance estimation framework. In this case, the performance estimation problem may be cast as follows,
\begin{align}\label{P_Qq}
\nonumber   \max & \ \frac{f^{N+1}-f^\star}{f^1-f^\star}\\
\nonumber \st \ & \{x^k, g^k, f^k\}\cup \{y^k, 0, f^\star\}  \ \text{satisfy inequality  \eqref{Int-c} for}\ k\in\{1, ..., N+1\} \\
&  \  x^{k+1}=x^k-\tfrac{1}{L} g^k, \ \  k\in\{1, ..., N\}   \\
\nonumber& \ f^k\geq f^\star  \ \  k\in\{1, ..., N\}\\
\nonumber  & \ f^k-f^\star \geq \tfrac{\mu_q}{2} \|x^k-y^k\|^2,  \ \  k\in\{1, ..., N+1\}\\
\nonumber  & \ \|x^k-y^k\|^2 \leq \|x^k-y^{k^\prime}\|^2,  \ \  k\in\{1, ..., N+1\}, k^\prime\in\{1, ..., N+1\}.
\end{align}
Since $x^{k+1}=x^k-\tfrac{1}{L} g^k$, we get $x^{k+1}=x^1-\tfrac{1}{L} \sum_{l=1}^k g^l$. Hence, problem \eqref{P_Qq} may be reformulated as follows,
\begin{align}\label{P2_Qq}
\nonumber   \max & \ \frac{f^{N+1}-f^\star}{f^1-f^\star}\\
\nonumber \st \ & \{x^1-\tfrac{1}{L} \sum_{l=1}^{k-1} g^l, g^k, f^k\}\cup \{y^k, 0, f^\star\}  \ \text{satisfy interpolation constraints \eqref{Int-c}}  \\
 & \ f^k\geq f^\star  \ \  k\in\{1, ..., N\}\\
\nonumber  & \ f^k-f^\star \geq \tfrac{\mu_q}{2} \|x^1-\tfrac{1}{L} \sum_{l=1}^{k-1} g^l-y^k\|^2,  \ \  k\in\{1, ..., N+1\}\\
\nonumber  & \ \| x^1-\tfrac{1}{L} \sum_{l=1}^{k-1} g^l-y^k\|^2 \leq \|x^1-\tfrac{1}{L} \sum_{l=1}^{k-1} g^l-y^{k^\prime}\|^2,  \ \  k, k^\prime\in\{1, ..., N+1\}.
\end{align}

The next theorem provides an upper bound for problem \eqref{P2_Qq} by using weak duality.

\begin{theorem}\label{Th.QF}
Let $f\in\mathcal{F}_{-L, L}(\mathbb{R}^n)$ and let $f$ have a quadratic functional growth on $X=\{x: f(x)\leq f(x^1)\}$ with $\mu_q\in(\tfrac{L}{2}, L)$. If $t_k=\tfrac{1}{L}$, $k\in\{1, ..., N\}$, then we have the following convergence rate for Algorithm \ref{Alg1},
\begin{align}\label{CR_Qg_pep}
f(x^{N+1})-f(x^1) \leq\tfrac{L}{\mu_q} \left( 2-\tfrac{2\mu_q}{L}\right)^N \left(f(x^1)-f^\star  \right).
\end{align}
\end{theorem}
\begin{proof}
The proof is analogous to that of Theorem \ref{Th.PL}. Without loss of generality, we may assume that $f^\star=0$. By some algebra, one can show that
{\small{
\begin{align*}
    & {f^{N+1}-f^\star}-\tfrac{L}{\mu_q} \left( 2-\tfrac{2\mu_q}{L}\right)^N\left( {f^1-f^\star} \right)+\sum_{j=1}^{N+1} \left(2^{N+1-j}(1-\tfrac{\mu_q}{L})^{N-1}\right)\times\\
    &\left(f^\star-f^j-\langle g^j, y^1-x^1+\tfrac{1}{L} \sum_{l=1}^{j-1} g^l\rangle -\tfrac{1}{2L}\|g^j\|^2+\tfrac{L}{4} \| y^1-x^1+\tfrac{1}{L} \sum_{l=1}^{j-1} g^l+\tfrac{1}{L} g^j\|^2 \right)+\\
 &  \sum_{i=2}^{N}\sum_{j=i}^{N+1} \left(2^{N+1-j}(\tfrac{\mu_q}{L})(1-\tfrac{\mu_q}{L})^{N-i}\right)\bigg(f^\star-f^j-\langle g^j, y^i-x^1+\tfrac{1}{L} \sum_{l=1}^{j-1} g^l\rangle -\\
 &\tfrac{1}{2L}\|g^j\|^2+\tfrac{L}{4} \| y^i-x^1+\tfrac{1}{L} \sum_{l=1}^{j-1} g^l+\tfrac{1}{L} g^j\|^2 \bigg)+\sum_{j=2}^{N} \left(2^{N+1-j}(1-\tfrac{\mu_q}{L})^{N-j}\right)\times\\
 &\left(f^j-f^\star-\tfrac{\mu_q}{2} \| y^j-x^1+\tfrac{1}{L} \sum_{l=1}^{j-1} g^l\|^2 \right)+\left(2^{N}(1-\tfrac{\mu_q}{L})^{N-1}+\tfrac{L}{\mu_q} \left( 2-\tfrac{2\mu_q}{L}\right)^N\right)\times\\
 &\left(f^{1}-f^\star-\tfrac{\mu_q}{2} \| y^{1}-x^1\|^2 \right)=-\left(\tfrac{L}{4}(1-\tfrac{\mu_q}{L})^{N-1}\|y^1-x^1+\tfrac{1}{L}\sum_{l=1}^{N+1} g^l\|^2\right)-\\
 &\sum_{i=2}^N\left( \tfrac{\mu_q}{4}(1-\tfrac{\mu_q}{L})^{N-i}\|y^i-x^1+\tfrac{1}{L}\sum_{l=1}^{N+1} g^l\|^2\right)\leq 0.
\end{align*}
}}
 Due to the weak duality, we get
 $$
 f^{N+1}-f^\star\leq \tfrac{L}{\mu_q} \left( 2-\tfrac{2\mu_q}{L}\right)^N\left( {f^1-f^\star} \right),
 $$
 and the proof is complete.
\end{proof}

By doing some calculus, one can verify the following inequality
\begin{align*}
\frac{2L^2-2(\mu_q-\tfrac{L}{2})^2}{2L^2+(\mu_q-\tfrac{L}{2})^2}\geq (2-\tfrac{2\mu_q}{L}), \ \  \mu_q\in(\tfrac{L}{2}, L).
\end{align*}
Hence, Theorem \ref{Th.QF} provides a tighter bound than \eqref{CR_Qg}.

\begin{definition}{\cite[Definition 1]{hinder2020near}}\label{Def.Qs}
Let $\gamma\in (0, 1]$ and $\mu_s\geq 0$. A function $f$ is called $(\gamma, \mu_s)$-quasar-convex on $X\subseteq \mathbb{R}^n$ with respect to  $x^\star\in\text{argmin}_{x\in\mathbb{R}^n} f(x)$ if
\begin{align}\label{QSC}
f(x)+\tfrac{1}{\gamma}\langle\nabla f(x), x^\star-x\rangle+\tfrac{\mu_s}{2}\| x^\star-x\|^2  \leq f^\star, \ \ \ \forall x\in X.
\end{align}
\end{definition}
The class of quasar-convex functions is large. For instance,  non-negative homogeneous functions are $(1, 0)$-quasar-convex on $ \mathbb{R}^n$. Indeed, if $f$ is non-negative homogeneous of degree $k\geq 1$, by the Euler identity, we have
$$
f(x)+\langle\nabla f(x), x^\star-x\rangle=(1-k) f(x)\leq 0, \ \forall x\in\mathbb{R}^n,
$$
where $x^\star=0$. In what follows, we list some convergence result concerning quasar-convex functions for
Algorithm \ref{Alg1}.

\begin{theorem}{\cite[Remark 4.3]{bu2020note}}\label{Shah}
Let $f$ be $L$-smooth and let $f$ be $(\gamma, \mu_s)$-quasar-convex on $X=\{x: f(x)\leq f(x^1)\}$. If  $t_1=\tfrac{1}{L}$ and if $x^2$ is from Algorithm \ref{Alg1}, then
\begin{align}
f(x^2)-f^\star\leq \left(1-\tfrac{\gamma^2\mu_s}{L}\right)\left(f(x^1)-f^\star  \right).
\end{align}
\end{theorem}

In the following theorem, we state the relationship between quasar-convexity and other concepts. Before we get to the theorem, we recall star convexity. A set $X$ is called star convex at $x^\star$ if
$$
\lambda x+(1-\lambda)x^\star \in X, \ \forall x\in X, \forall\lambda\in [0, 1].
$$
\begin{theorem}\label{rel2}
Let $x^\star$ be the unique solution of problem \eqref{P} and let $X=\{x: f(x)\leq f(x^1)\}$. If $X$  is star convex at $x^\star$, then we have the following implications:
\begin{enumerate}[i)]
\item
    \eqref{QSC} $\Rightarrow$ \eqref{Qq_r} with  $ \mu_g=\tfrac{\mu_s\gamma}{2}+\tfrac{\mu_s\gamma^2}{4}$.
 \item
    \eqref{Qq_r} $\Rightarrow$ \eqref{QSC} with $\mu_s=\ell-\tfrac{L}{2}$ and $\gamma=\tfrac{\mu_g}{\ell}$ for each $\ell\in (\max(\tfrac{L}{2}, \mu_g), \infty)$.
    \item
    \eqref{QSC} $\Rightarrow$ \eqref{PL} with  $ \mu_p=\mu_s\gamma^2$.
\end{enumerate}
\end{theorem}
\begin{proof}
The proof of $i)$ is similar in spirit to the proof of Theorem 1 in \cite{necoara2019linear}. Let $x\in X$. By  the fundamental theorem of calculus and \eqref{QSC},we have
\begin{align*}
f(x)-f(x^\star)&=\int_0^1 \tfrac{1}{\lambda}\langle \nabla f(\lambda x+(1-\lambda)x^\star), \lambda x+(1-\lambda)x^\star-x^\star\rangle d\lambda\\
&\geq \int_0^1 \tfrac{\gamma}{\lambda}\left( f(\lambda x+(1-\lambda)x^\star)-f(x^\star)+\tfrac{\mu_s\lambda^2}{2} \|x-x^\star\|^2\right)d\lambda\\
& \geq \tfrac{\gamma\mu_s}{4}\|x-x^\star\|^2,
\end{align*}
where the last inequality follows from the global optimality of $x^\star$. By summing $f(x)-f(x^\star)\geq \tfrac{\gamma\mu_s}{4}\|x-x^\star\|^2$ and \eqref{QSC}, we get the desired inequality.
Now, we prove part $ii)$. Let $x\in\mathbb{R}^n$ and $\ell\in (\max(\tfrac{L}{2}, \mu_g), \infty)$. By \eqref{L_E1}, we have
\begin{align}\label{L-C}
 f(x)\leq f(x^\star)+\tfrac{L}{2} \|x-x^\star\|^2.
\end{align}
By using \eqref{L-C} and \eqref{Qq_r}, we get
$$
f(x)+(\tfrac{\ell}{\mu_g})\langle \nabla f(x), x^\star-x\rangle+(\ell-\tfrac{L}{2})\|x-x^\star\|^2 \leq f(x^\star).
$$
For the proof of $iii)$, we refer the reader to \cite[Lemma 3.2]{bu2020note}.
\end{proof}

By combining Theorem \ref{Th.PL} and Theorem \ref{rel2}, under the assumptions of Theorem \ref{Shah}, one can get the following convergence rate for Algorithm \ref{Alg1} with $t_1=\tfrac{1}{L}$,
$$
f(x^2)-f^\star\leq \left(\frac{2L-2\mu_s\gamma^2}{2L+\mu_s\gamma^2}\right)\left(f(x^1)-f^\star\right),
$$
which is tighter the bound given in Theorem \ref{Shah}.

\section{Concluding remarks}
This paper studied the convergence rate of the gradient method with fixed step lengths for smooth functions satisfying the  P\L\ inequality. We gave a new convergence rate, which is sharper than known bounds in the literature.  One important question which remains to be addressed is  the computation of the tightest bound for Algorithm \ref{Alg1}. Moreover, the performance analysis of fast gradient methods, like Algorithm \ref{Alg2}, for these classes of functions may also be of interest.

\bibliographystyle{spmpsci}      
\bibliography{references}

\end{document}